\documentclass[10pt, reqn]{article}
\usepackage{amsfonts}
\usepackage{amsmath}
\usepackage[utf8]{inputenc}
\usepackage[english]{babel}
\usepackage{amssymb}
\usepackage{units}
\usepackage{graphicx}
\usepackage{subcaption}
\usepackage{esvect}
\usepackage{wrapfig}
\usepackage[usenames, dvipsnames]{color}
\usepackage{multicol}
\usepackage{setspace}
\usepackage{tikz}
\usepackage{tikz-cd}
\usepackage{amsthm}
\usepackage{algorithm}
\usepackage{nccmath}
\usepackage{enumerate}
\usepackage{xcolor,colortbl}
\usepackage{scalerel}
\usetikzlibrary{calc, positioning}
\usepackage{bigstrut}

\usepackage{xpatch}
\makeatletter
\xpatchcmd{\@thm}{\thm@headpunct{.}}{\thm@headpunct{}}{}{}
\makeatother
\makeatletter
\def\Ddots{\mathinner{\mkern1mu\raise\p@
\vbox{\kern7\p@\hbox{.}}\mkern2mu
\raise4\p@\hbox{.}\mkern2mu\raise7\p@\hbox{.}\mkern1mu}}
\makeatother
\usepackage[rightcaption]{sidecap}
\usepackage{tabu}
\sidecaptionvpos{figure}{c}
\usepackage[margin=2.5cm]{geometry}
\makeatletter
\def\@biblabel#1{}
\makeatother
\usepackage{etoolbox} 
\makeatletter
\patchcmd{\subsubsection}{\itshape}{\bfseries\itshape}{}{}
\makeatother
\usepackage{pythonhighlight}
\graphicspath{ {../Figures/} }
\usepackage{array}
\newcolumntype{L}{>{$}l<{$}}
\newcolumntype{R}{>{$}r<{$}}
\addtocontents{toc}{~\hfill\textbf{Page}\par}

\makeatletter
\def\l@section{\@tocline{2}{0pt}{0pc}{2.5pc}{}}
\def\l@subsection{\@tocline{2}{0pt}{2.5pc}{3.5pc}{}}
\def\l@subsubsection{\@tocline{2}{0pt}{5.25pc}{7.5pc}{}}
\usepackage{tocbasic}
\usepackage[skins]{tcolorbox}
\newtcolorbox{myframe}[2][]{%
  enhanced,colback=white,colframe=black,coltitle=black,
  sharp corners,boxrule=0.4pt,
  fonttitle=\itshape,
  attach boxed title to top left={yshift=-0.3\baselineskip-0.4pt,xshift=2mm},
  boxed title style={tile,size=minimal,left=0.5mm,right=0.5mm,
    colback=white,before upper=\strut},
  title=#2,#1
}
\makeatletter
\renewcommand*{\@biblabel}[1]{\textbf{\hfill#1.}}
\makeatother

\RequirePackage[colorlinks=true,
    pdfstartview=FitV,
    linkcolor=black,
    citecolor=black,
    urlcolor=black,
    breaklinks=true,
    pdfencoding=auto]{hyperref}
\RequirePackage{aliascnt}
\RequirePackage{cleveref}
\newtheorem{thm}{Theorem}[section]
    \crefformat{thm}{\bfseries theorem~#2#1#3}
    \Crefformat{thm}{\bfseries Theorem~#2#1#3}
\newtheorem{corollary}[thm]{Corollary}
    \crefformat{corollary}{\bfseries corollary~#2#1#3}
    \Crefformat{corollary}{\bfseries Corollary~#2#1#3}
\newtheorem{lemma}[thm]{Lemma}
    \crefformat{lemma}{\bfseries lemma~#2#1#3}
    \Crefformat{lemma}{\bfseries Lemma~#2#1#3}
\newtheorem{prop}[thm]{Proposition}
    \crefformat{prop}{\bfseries proposition~#2#1#3}
    \Crefformat{prop}{\bfseries Proposition~#2#1#3}
\theoremstyle{definition}
\newaliascnt{defin}{thm}
\newtheorem{defin}[thm]{Definition}
    \crefformat{defin}{\bfseries definition~#2#1#3}
    \Crefformat{defin}{\bfseries Definition~#2#1#3}
\newtheorem{remark}[thm]{Remark}
    \crefformat{remark}{\bfseries remark~#2#1#3}
    \Crefformat{remark}{\bfseries Remark~#2#1#3}
\newtheorem{exa}[thm]{Example}
    \crefformat{exa}{\bfseries example~#2#1#3}
    \Crefformat{exa}{\bfseries Example~#2#1#3}

    \crefformat{conjecture}{\bfseries conjecture~#2#1#3}
    \Crefformat{conjecture}{\bfseries Conjecture~#2#1#3}
    \usepackage{fancyhdr}
\Crefformat{figure}{#2\bfseries\figurename~#1#3}
\Crefformat{table}{#2\bfseries\tablename~#1#3}
    \RequirePackage{titlesec}
\titleformat{\section}{\centering\large\bfseries}{\thesection}{1em}{}
\titleformat{\subsection}{\bfseries}{\thesubsection}{1em}{}
\titleformat{\subsubsection}{\normalsize\bfseries}{\thesubsubsection}{1em}{}
    	\title{\Large Combinatorial Approach to ABV-packets for $\mathbf{GL_n}$}
     \date{\vspace{-2ex}}
	\author{ Connor Riddlesden\thanks{Present Institution: Eindhoven University of Technology, the Netherlands, \url{c.d.riddlesden@tue.nl}}\,, University of Lethbridge}
\begin{document}
\maketitle
\begin{abstract}
There exists a significant conjecture in the local Langlands correspondence that A-packets are ABV-packets. For the case $G=GL_n$, the conjecture reduces to ABV-packets for orbits of Arthur type in $GL_n$ being singletons, which is a specialisation of the wider conjecture known as the Open-Orbit conjecture. In this paper, we will prove the reduced conjecture since there exists a nice combinatorial description. The result first appeared in the associated Master's thesis, however we aim to use a slightly more simplified and succinct approach in this paper using results of Knight and Zelevinskii. We will also prove the partial ordering relation associated to the conjecture for multisegments of ladder type. \\
\end{abstract}
This article contains material originally appearing in my Master's thesis \cite{riddlesden2022combinatorial}.
\section{Background and Motivation}
We are interested in proving the conjecture that ABV-packets for orbits of Arthur type are singletons when $G = GL_n$. This conjecture, as introduced in \cite{Cunningham2}, is a specialisation of the wider conjecture known as the Open-Orbit conjecture. While its general framework is of relevance to the local Langlands correspondence, a particularly prominent feature of this discussion is the combinatorial description using multisegments, first provided by M\oe glin and Waldspurger in \cite{Moeglin}. As proved in \cite[Proposition  2.3.7.]{riddlesden2022combinatorial}, the orbits of the objects categorising the ABV-packet are classified by associated ranks $r_{i,j}$, which naturally construct rank triangles as follows\\ 
\[
\begin{array}{ccccccccc}
r_{1, 1} & & r_{2, 2} & & r_{3, 3} & & r_{4, 4} & & \cdots \\
& r_{1, 2} & & r_{2, 3} & & r_{3, 4} & & \ddots & \\
& & r_{1, 3} & & r_{2, 4} & &  \ddots & & \\
& & & r_{1, 4} & &  \ddots & & & \\
& & & & \ddots & & & & \\
\end{array}.
\]\\ \\
It is also a natural question to want to study the dual orbit which once again will be classified by its ranks which are denoted by $\widetilde{r_{i,j}}$. 

Further to this, we can define a natural partial ordering on rank triangles. Let $C$ and $D$ denote respective rank triangles with respective values $c_{i,j}$ and $d_{i,j}$, then we say that $C \leq D$ if for all indexes $(i,j)$ we have the property $c_{i,j} \leq d_{i,j}$. We are especially interested in the case in which for every index $i$ we have the property $c_{i,i} = d_{i,i}$, that is, when the top rows of the rank triangles are equal.

We will now present a purely combinatorial interpretation of the rank triangle.
\begin{defin}
Let us define a \emph{segment} to be a non-empty set of consecutive integers $$ \Delta = (b, b+1, \dots, e-1, e).$$
Then a \emph{multisegment} will be a collection of segments $$ \alpha = \{ \Delta_1, \Delta_2, \dots, \Delta_r \},$$ where each segment is indexed by $i$, for $1 \leq i \leq r$, to differentiate between possible duplicates of segments. 
\end{defin}
Similar to the ranks, we will denote the dual multisegment associated to $\alpha$ by $\tilde{\alpha}$. For ease of notation, given a segment $\Delta = (b, b+1, \dots , e-1, e)$, then we say the \emph{base} of $\Delta$ is $b$ and the \emph{end} of $\Delta$ is $e$. More precisely, the base of a segment corresponds to the smallest integer and the end of the segment to the largest integer in the segment. 

The following proposition provides an incredibly quick method for computing for constructing the multisegment associated to the rank triangle.
\begin{prop} \label{multiplicity}
The multiplicity of a segment $(i, i+1, \dots, j-1, j)$ in the multisegment is denoted $m_{i, j}$, and given by $$ m_{i, j} = r_{i, j} - r_{i-1, j} - r_{i, j+1} + r_{i-1, j+1}. $$ Note that we assume that the rank is equal to $0$ if it is not defined inside the rank triangle.
\end{prop}
The ranks $r_{i,j}$ are then simply the number of segments $[k,l]$ contained in the associated multisegment $\alpha$ such that $k \leq i$ and $j \leq l$.

There also exists a partial ordering relation for multisegments which we can study following the introduction of the following action between any two segments of the multisegment.
\begin{prop} \label{Prop:Actions}
Let $\Delta_1$ and $\Delta_2$ be any two segments in an arbitrary multisegment $\alpha$, then we can construct a new multisegment $\beta$ by replacing each $\Delta_1$ and $\Delta_2$ in $\alpha$ with respectively
\[ \left\{ \begin{array}{ll}
        \Delta_1 \cap \Delta_2 \text{ and } \Delta_1 \cup \Delta_2, & \mbox{if $ \Delta_1 \cap \Delta_2 \neq \emptyset$ and $\Delta_1 \neq \Delta_2$};\\
        \Delta_1 \cup \Delta_2, & \mbox{else if $\Delta_1 \cup \Delta_2$ is a segment}; \\ \Delta_1 \text{ and } \Delta_2, & \mbox{otherwise}. \\ \end{array} \right.  \] then we have that $ \alpha \leq \beta$.
\end{prop}
Note we will refer to the first action as the \emph{union intersection} and the second as \emph{conjunction}. The third action will simply leave the multisegment unchanged. 

The aim of this paper will be to study multisegments $\alpha$ with specific inherent properties, mainly a partial ordering relation in which for any multisegment $\beta$ such that $\alpha \leq \beta $ and $\tilde{\alpha} \leq \tilde{\beta}$ then $\alpha = \beta$. The relation with the rank triangle description then follows:

\begin{prop}[\cite{riddlesden2022combinatorial}, Corollary 2.5.7.] \label{Cor:BoundaryRTandMS}
Suppose that $\alpha$ and $\beta$ are multisegments with respective corresponding conjugacy classes denoted by $C$ and $D$ (with identical top rows). Then there exists a partial ordering between the two multisegments if and only if there exists a partial ordering on their corresponding conjugacy classes, that is, $$ \alpha \leq \beta \,\,\, \text{   if and only if   } \,\,\, C \leq D .$$
\end{prop}

The complete background and motivation for this paper can be found in Chapter 2 of the associated Master's thesis \cite{riddlesden2022combinatorial}. In addition to this, the connection between the local Langlands correspondence, the conjecture of \cite{Cunningham2}, and the combinatorial statement in \Cref{Thm:ManySimpleA=B} is explained by Cunningham and Ray in \cite{Ray}.

We begin in \Cref{Sec:MW-alg} by presenting a method for calculating the dual multisegment from the original multisegment using the work of M\oe glin and Waldspurger. We will then discuss an alternative approach first presented by Knight and Zelevinskii using a network implementation. Finally in \Cref{Sec:Partial}, we study the partial ordering relation for multisegments of various types including those associated to Langlands parameters of Arthur type in \Cref{Thm:ManySimpleA=B}. This partial ordering will then lead to the proof of the main result, the conjecture (\Cref{Cor:Arthur}):
\begin{center}\vspace{5pt}
    \emph{For $G=GL_n$, ABV-packets for orbits of Arthur type are singletons and consequently, \\ ABV-packets for orbits of Arthur type are A-packets}. \vspace{5pt}
\end{center}
The proof of these results for $G=GL_n$ is heavily reliant on the orbits being naturally classified by their ranks and the multisegment description, and the duality algorithms. This classification and the associated algorithms does not hold true in general for different choices of $G$, nor does there currently exist any analogous combinatorial descriptions. Therefore, a more representation theoretic approach may be needed to generalise this result to other groups. That being said developing this approach for $GL_n$ is still valuable and could provide helpful insight in the generalisation.

\section{M\oe glin-Waldspurger Algorithm}
\label{Sec:MW-alg}
We can now begin to introduce the M\oe glin-Waldspurger algorithm which will compute the multisegment associated to the dual orbit from the original multisegment. The algorithm will use the following relation between segments:

\begin{defin}[\cite{Moeglin}] \label{Precedes}
Given any two segments $\Delta_1 = (b_1, \dots, e_1)$ and $\Delta_2 = (b_2, \dots, e_2)$, then we say that $\Delta_1$ \emph{precedes} $\Delta_2$ if $b_1 < b_2$, $e_1 < e_2$, and $b_2 \leq e_1 +1$.
\end{defin}

\newpage
 \noindent{\bfseries Algorithm : M\oe glin-Waldspurger} \cite{Moeglin}
Given a multisegment $\alpha$ with maximum value $e$ then we can compute the multisegment $\tilde{\alpha}$ associated to the dual orbit as follows:
 \begin{enumerate}
 \item Let $m = e$ be the maximum value in the multisegment and set $\Delta_m$ to be the shortest segment whose maximal value is $m$.
 \item If there does not exist a segment that precedes $\Delta_m$ whose maximal value is $m-1$, then go to step 5.
 \item Amongst the segments that precede $\Delta_m$ whose maximal value is $m-1$, select $\Delta_{m-1}$ to be the shortest such segment.
 \item Set $m := m-1$ and return to step 2.
 \item For each segment $\Delta_i$ for $m \leq i \leq e$ remove the maximal value $i$ from this segment. Following the removal of these end values, let us denote the new multisegment to be $\alpha'$.
 \item The dual segment formed will be $\Delta'=(m, \dots, e)$.
 \end{enumerate}
 Generating the segment $\Delta'$ will be from here forward referred to as a single iteration of the algorithm. To find the complete dual multisegment $\tilde{\alpha}$ one will need to continue this process recursively using $$ \tilde{\alpha} = \left\{ \Delta', \widetilde{(\alpha')} \right\}.$$ M\oe glin and Waldspurger then prove in \cite[Theorem 13]{Moeglin} that $\tilde{\alpha}$ will be equal to the multisegment of the dual representation found by the Zelevinskii involution.
 
The length of the segment $\Delta = (b, b+1, \dots , e-1, e)$ will be given by $e-b+1$. To simplify the notation we will now represent the segment $(i,i+1, \dots, j)$ by $[i, j]$, or if the segment is a singleton $(i)$ then $[i]$.

\begin{prop} \label{Prop:IncreasingLength}
During each iteration of the M\oe glin-Waldspurger algorithm the preceding segments will be chosen in increasing length.
\end{prop}

\begin{proof}
Let $m$ be the integer chosen from the segment $\Delta_m$ by the M\oe glin-Waldspurger. If $m$ is not the base of the segment generated by the algorithm, then there exists a segment $\Delta_{m-1}$ that precedes $\Delta_m$ with end value $m-1$. Let $b_m$ denote the base value of the segment $\Delta_m$, then by the precedes condition $b_{m-1}<b_m,$ which implies that $$m - b_m +1< m-b_{m-1}+1,$$ and $$m - b_m +1\leq (m-1)-b_{m-1}+1.$$
Thus the length of the preceding segments will be chosen in increasing length.
\end{proof}

\subsection{Network Implementation}
\label{Sub:Network}
In \cite{Knight}, Knight and Zelevinskii show that the Zelevinskii involution (and hence the M\oe glin-Waldspurger algorithm) can be implemented using a network description. Given a multisegment $\alpha=\{\Delta_1, \dots, \Delta_r\}$ with corresponding ranks $r_{i,j}$, then they use the following procedure to implement the Zelevinskii Involution:
 \begin{enumerate}
 \item For each segment $\Delta \in \alpha$ and integer $a \in \Delta$ create a vertex $v_{\Delta, a}$. We will then denote the set of vertices associated to an integer $a$ to be $V_a=\{v_{\Delta, a} \mid \Delta \in \alpha, a \in \Delta \}$.
 \item For all segments $\Delta_n, \Delta_m \in \alpha$, if $\Delta_n$ precedes $\Delta_m$ then for all integers $a$ such that $a \in \Delta_a$ and $a+1 \in \Delta_b$ add an edge from $v_{\Delta_m, a+1}$ to $v_{\Delta_n, a}$.
 \item The ranks $\widetilde{r_{i,j}}$ corresponding to the dual multisegment will be the maximum number of vertex-disjoint paths from the set of vertices $V_j$ to the set of vertices $V_i$.
 \end{enumerate}
 Note it possible to implement this method for searching for the maximum number of vertex-disjoint paths by searching for the maximum flow on an equivalent network. We instead use the following procedure:
 \begin{enumerate}
 \item Split each vertex $v$ in the graph into the nodes $v^0$, $v^1$, and add an edge of capacity $1$ from $v^0$ to $v^1$.
 \item Replace each other edge $(u, v)$ in the graph with an edge of capacity $1$ from $u^1$ to $v^0$.
 \item Add additional nodes $s$ and $t$.
 \item For each rank $\widetilde{r_{i,j}}$:
 \begin{enumerate}
 \item Add edges of capacity $1$ from $s$ to each node in the set $V^0_j$, and from each node in the set $V^1_i$ to $t$. 
 \item The rank $\widetilde{r_{i,j}}$ will then correspond to the maximum flow from $s$ to $t$.
 \end{enumerate}
 \end{enumerate}

\section{The Partial Ordering Relation on Families of Multisegments} \label{Sec:Partial}
We have previously discussed a method for computing the dual of a multisegment, however we are yet to study properties and relations satisfied by specific families of $G$, mainly that of the open-orbit conjecture. Thus we now devote this section to further our investigation into the dual group. We will use a number of natural formations of multisegments in this study and overall conclude a significant property for those multisegments of Arthur type. Given a multisegment $\alpha$ then we say that $\alpha$ satisfies the \emph{partial ordering relation} when:
$$ \abovedisplayskip=12pt \text{For all multisegments $\beta$ such that  } \alpha \leq \beta \, \, \, \text{and} \,\,\, \tilde{\alpha} \leq \tilde{\beta},  \,\,\, \text{we find } \alpha = \beta. \belowdisplayskip=12pt $$
We will use the following six numerical invariants to study the partial ordering relation:
\begin{enumerate}[i)]
\item $e_{\alpha} := \max(\alpha)$;
\item $L_{\alpha} :=$ Length of the longest segment in $\alpha$;
\item $n_{\alpha} :=$ Number of segments in $\alpha$;
\item $c_{\alpha} :=$ Minimum number of segments in which $\cup_{\Delta \in \alpha} \Delta$ constructs;
\end{enumerate}
Let us denote the segments generated by the minimal formation of $\cup_{\Delta \in \alpha} \Delta$ to be $\Delta^1, \dots, \Delta^{c_{\alpha}}$.
\begin{enumerate}[i)]
\setcounter{enumi}{4}
\item $S_{\alpha} := \sum^{c_{\alpha}}_{i=1} \left| \Delta^i \right| $;
\item $C_{\alpha} :=$ Maximum number of components in a decomposition $\alpha = \bigsqcup_i \alpha_i $ for which $\tilde{\alpha} = \bigsqcup_i \widetilde{\alpha_i} $.
\end{enumerate}

Note that if $c_{\alpha}=1$, then there exists a single segment in the minimal formation of $\cup_{\Delta \in \alpha} \Delta$ . In this case, let us define $S_{\alpha} := \left| \cup_{\Delta \in \alpha} \Delta \right|$.

Note that studying the case in which $c_{\alpha} > 1$ is unnecessary as there is effectively no interaction between the individual components, so we could simply consider them individually. Thus we will call a multisegment $\alpha$ \emph{connected} if $c_{\alpha}=1$. Further, we will call a multisegment $\alpha$ \emph{irreducible} when $C_{\alpha} = 1$ and hence the decomposition $\tilde{\alpha} = \bigsqcup_i \widetilde{\alpha_i} $ into $C_{\alpha}$ number of components will be the \emph{irreducible decomposition}.

\begin{lemma}[\cite{riddlesden2022combinatorial}, Lemma 4.2.2.]\label{Basic}
Let $\alpha, \beta$ be multisegments and $\tilde{\alpha}, \tilde{\beta}$ their respective dual multisegments. Then 
\begin{enumerate}[i)]
\item For any two multisegments $\alpha$, $\beta$ with isomorphic quiver representations, $c_{\alpha} = c_{\beta}$ and $S_{\alpha} = S_{\beta}$.
\item If $\alpha \leq \beta$ then $L_{\alpha} \leq L_{\beta}$.
\item If $\alpha \leq \beta$ then $n_{\alpha} \geq n_{\beta}$.
\item $n_{\tilde{\alpha}} \geq L_{\alpha}$ and $n_{\alpha} \geq L_{\tilde{\alpha}}$.
\item $C_{\alpha} \geq c_{\alpha}$.
\end{enumerate}
\end{lemma}

\begin{lemma}[\cite{riddlesden2022combinatorial}, Lemma  4.2.3.] \label{EqualLN}
If $L_{\tilde{\alpha}} = n_{\alpha}$, $\alpha \leq \beta$ and $\tilde{\alpha} \leq \tilde{\beta}$, then $n_{\alpha} = n_{\beta} = L_{\tilde{\alpha}}  = L_{\tilde{\beta}}$.   
\end{lemma}

\subsection{Simple Multisegments}

\begin{defin} \label{Def:Simple}
A multisegment $\alpha$ is \emph{simple} if it has the form: $$ \alpha = \{ [b,\dots, e], [b+1,\dots, e+1], \dots , [b+n-1,\dots, e+n-1] \}.  $$
\end{defin}
\begin{exa} \label{Exa:Simple}
The first example is a \emph{simple multisegment} for which each segment has the same length and the end values reduce by one each time. \\

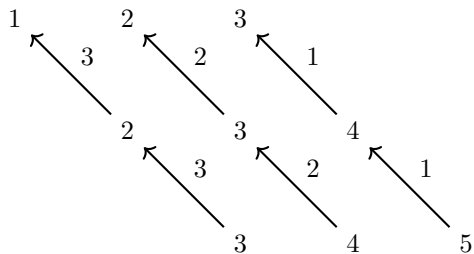
\begin{figure}[H]
 \hspace*{4cm}
\begin{tikzpicture}[node distance={15mm}, thick, main/.style = {draw=none}, transform canvas={scale=1}] 
\node[] (1) {$1$}; 
\node[main] (2) [right of=1] {$2$}; 
\node[main] (3) [right of=2]{$3$}; 
\node[main] (5) [below of=2]{$2$}; 
\node[main] (6) [right of=5] {$3$}; 
\node[main] (7) [right of=6]{$4$}; 
\node[main] (9) [below of=6]{$3$};  
\node[main] (10) [right of=9]{$4$};  
\node[main] (11) [right of=10]{$5$};  
\draw[->,black] (11) -- (7) node[midway,above right] {1}; 
\draw[->,black] (7) -- (3) node[midway,above right] {1}; 
\draw[->,black] (10) -- (6) node[midway,above right] {2}; 
\draw[->,black] (6) -- (2) node[midway,above right] {2}; 
\draw[->,black] (9) -- (5) node[midway,above right] {3}; 
\draw[->,black] (5) -- (1) node[midway,above right] {3}; 
\end{tikzpicture}
\vspace{3.5cm}
\caption{The M\oe glin-Waldspurger algorithm on a simple multisegment.}\label{Fig:MWSimple}
\end{figure}
The dual of $\{(1,2,3), (2,3,4), (3,4,5) \}$ is $\{(1,2,3), (2,3,4), (3,4,5) \}$.
\end{exa}

Given a simple multisegment, then we can index the diagonals $D_i$ given by the natural ordering where $i$ denotes their maximum values. This partitioning into the diagonals is illustrated in \Cref{Tab:Diagonal} for the simple multisegment in \Cref{Fig:Simple}.  \\ \\ \\

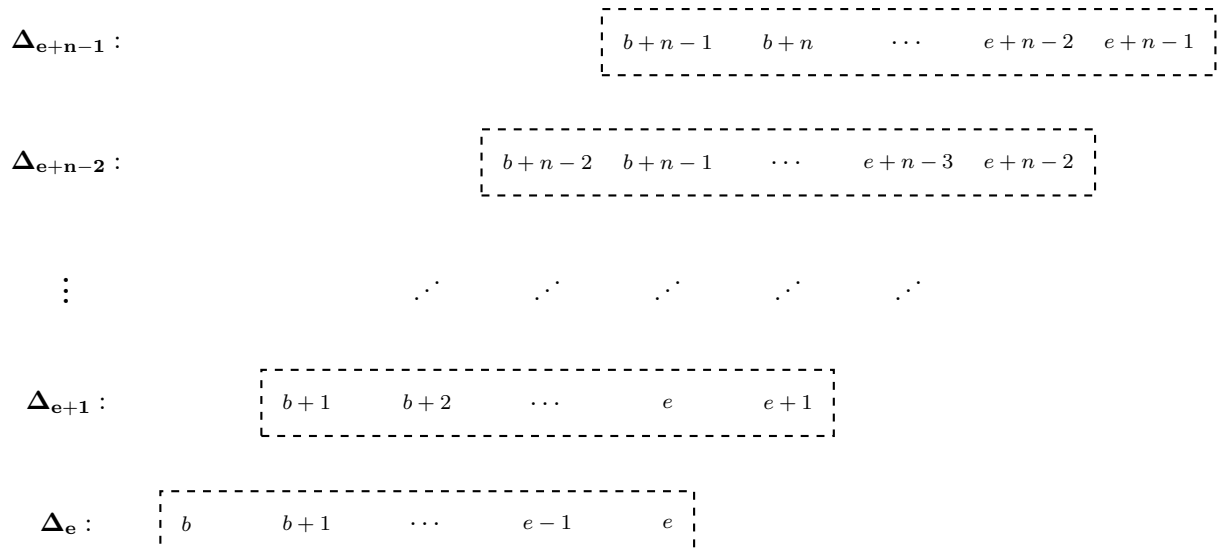
\begin{figure}[H]
\vspace{-1.25cm}
 \hspace*{0.8cm}
\begin{tikzpicture}[node distance={16mm}, thick, main/.style = {draw=none}, transform canvas={scale=1}] 
\node[] (1) {$\mathbf{\Delta_{e+n-1}:}$}; 
\node[main] (2) [below of=1] {$\mathbf{\Delta_{e+n-2}:}$}; 
\node[main] (3) [below of=2]{\textbf{\vdots}}; 
\node[main] (4) [below of=3] {$\mathbf{\Delta_{e+1}:}$}; 
\node[main] (5) [below of=4]{$\mathbf{\Delta_e:}$}; 
\node[main] (6) [right of=5] {\footnotesize$b$}; 
\node[main] (7) [right of=6]{\footnotesize$b+1$}; 
\node[main] (8) [right of=7] {$\dots$};
\node[main] (9) [right of=8]{\footnotesize$e-1$};  
\node[main] (10) [right of=9]{\footnotesize$e$};  
\node[main] (11) [above of=7]{\footnotesize$b+1$};  
\node[main] (12) [right of=11]{\footnotesize$b+2$};  
\node[main] (13) [right of=12]{$\dots$};  
\node[main] (14) [right of=13]{\footnotesize$e$};  
\node[main] (15) [right of=14]{\footnotesize$e+1$};  

\node[main] (16) [above of=12] {$\reflectbox{$\ddots$}$}; 
\node[main] (17) [right of=16]{$\reflectbox{$\ddots$}$}; 
\node[main] (18) [right of=17] {$\reflectbox{$\ddots$}$};
\node[main] (19) [right of=18]{$\reflectbox{$\ddots$}$};  
\node[main] (20) [right of=19]{$\reflectbox{$\ddots$}$};  
\node[main] (21) [above of=17]{\footnotesize$b+n-2$};  
\node[main] (22) [right of=21]{\footnotesize$b+n-1$};  
\node[main] (23) [right of=22]{$\dots$};  
\node[main] (24) [right of=23]{\footnotesize$e+n-3$};  
\node[main] (25) [right of=24]{\footnotesize$e+n-2$};  
\node[main] (26) [above of=22] {\footnotesize$b+n-1$}; 
\node[main] (27) [right of=26]{\footnotesize$b+n$}; 
\node[main] (28) [right of=27] {$\dots$};
\node[main] (29) [right of=28]{\footnotesize$e+n-2$};  
\node[main] (30) [right of=29]{\footnotesize$e+n-1$};
\draw[black,thick,dashed] ($(26.north west)+(-0.15,0.2)$)  rectangle ($(30.south east)+(0.15,-0.2)$);
\draw[black,thick,dashed] ($(21.north west)+(-0.15,0.2)$)  rectangle ($(25.south east)+(0.15,-0.2)$);
\draw[black,thick,dashed] ($(11.north west)+(-0.15,0.2)$)  rectangle ($(15.south east)+(0.15,-0.2)$);
\draw[black,thick,dashed] ($(6.north west)+(-0.15,0.2)$)  rectangle ($(10.south east)+(0.15,-0.2)$);
\end{tikzpicture}
\vspace{7.25cm}
\caption{The simple multisegment, $\alpha$.}\label{Fig:Simple}
\end{figure}

\vspace{0.25cm}
\begin{table}[h!]
\centering
\begin{tabular}{|ccccc|}
\hline
\multicolumn{5}{|c|}{\textbf{Diagonals}} \\ \hline
\multicolumn{1}{|c|}{$\mathbf{D_b}$} & \multicolumn{1}{c|}{$\mathbf{D_{b+1}}$} & \multicolumn{1}{c|}{$\mathbf{\cdots}$} & \multicolumn{1}{c|}{$\mathbf{D_{e-1}}$} & $\mathbf{D_e}$ \\ \hline
\multicolumn{1}{|c|}{$b$} & \multicolumn{1}{c|}{$b+1$} & \multicolumn{1}{c|}{$\cdots$} & \multicolumn{1}{c|}{$e-1$} & $e$ \\
\multicolumn{1}{|c|}{$b-1$} & \multicolumn{1}{c|}{$b$} & \multicolumn{1}{c|}{$\cdots$} & \multicolumn{1}{c|}{$e-2$} & $e-1$ \\
\multicolumn{1}{|c|}{$\vdots$} & \multicolumn{1}{c|}{$\vdots$} & \multicolumn{1}{c|}{$\vdots$} & \multicolumn{1}{c|}{$\vdots$} & $\vdots$ \\
\multicolumn{1}{|c|}{$-e+1$} & \multicolumn{1}{c|}{$-e+2$} & \multicolumn{1}{c|}{$\cdots$} & \multicolumn{1}{c|}{$-b$} & $-b+1$ \\
\multicolumn{1}{|c|}{$-e$} & \multicolumn{1}{c|}{$-e+1$} & \multicolumn{1}{c|}{$\cdots$} & \multicolumn{1}{c|}{$-b-1$} & $-b$ \\ \hline
\end{tabular}\vspace{0.3cm}
\caption{The simple multisegment $\alpha$ partitioned into its diagonals.}
\label{Tab:Diagonal}
\end{table}
\vspace{0.25cm}

To compute the maximum flow for each pair $(i,j)$ and hence find $\widetilde{r_{i,j}}$ one must push flow through the edges of the network associated to the individual diagonals. More specifically, for each pair $(i,j)$ and each diagonal $D_k$ one checks if the diagonal $D_k$ contains both integers $i$ and $j$. If both integers are contained, then flow must be pushed through the edges associated to the diagonal in the network, otherwise no flow will be pushed. Once each diagonal has been checked then the maximum flow will have been pushed and hence $\widetilde{r_{i,j}}$ will have been found.

The numerical invariants defined earlier in the section now provide a number of important properties for the family of simple multisegments.

\begin{prop}[\cite{riddlesden2022combinatorial}, Proposition  4.2.7.] \label{Prop:SimpleFacts}
If $\alpha$ is a simple multisegment then $n_{\tilde{\alpha}}=L_{\alpha}$.
\end{prop}

\begin{lemma}[\cite{riddlesden2022combinatorial}, Lemma  4.2.8.]\label{EqualAB}
If $\alpha$ is simple, $\alpha \leq \beta$ and $L_{\alpha} = L_{\beta}$, then $\alpha = \beta$.
\end{lemma}

We can now study the partial ordering relation for the first family of multisegments, simple multisegments, by using these properties.

\begin{thm}\label{Thm:Simple}
Let $\alpha$ be a simple multisegment. For any multisegment $\beta$ which satisfies the conditions $\alpha \leq \beta$ and $\tilde{\alpha} \leq \tilde{\beta}$, then $\alpha = \beta$. 
\end{thm}

\begin{proof}
Firstly, by assumption the multisegment $\alpha$ is simple so using \Cref{Prop:SimpleFacts} we know that $L_{\alpha} = n_{\tilde{\alpha}}$, and hence \Cref{EqualLN} therefore implies that $ L_{\alpha}= L_{\beta}$. Finally, using \Cref{EqualAB} we can conclude that $\alpha = \beta$.
\end{proof}
Therefore we have proved that the partial ordering relation will be satisfied for all simple multisegments.

\subsection{Ladder Multisegments}
We will now study broader family of multisegments in ladder multisegments, for which there exists a natural ordering between each of the segments.
\begin{defin}
We say that a multisegment $\alpha$ is a \emph{ladder multisegment} if it has the form: $$\alpha = \{ \Delta_1, \dots , \Delta_{n_{\alpha}} \},$$ where if we write $\Delta_i = [b_i, e_i]$ then for each $i<j$ we must have $b_i < b_j$ and $e_i < e_j$.
\end{defin}
\begin{exa}  \label{Exa:QFM}
This example is a \emph{ladder multisegment} for which there exists a complete ordering of the segments based around their base and end values. \\
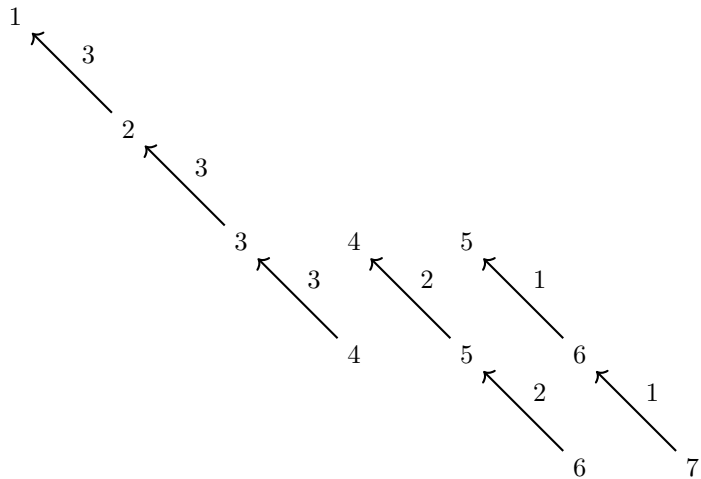
\begin{figure}[H]
 \hspace*{3cm}
\begin{tikzpicture}[node distance={15mm}, thick, main/.style = {draw=none}, transform canvas={scale=1}] 
\node[] (1) {$1$}; 
\node[main] (2) [right of=1] {\phantom{$1$}}; 
\node[main] (3) [below of=2]{$2$}; 
\node[main] (4) [right of=3] {\phantom{$1$}}; 
\node[main] (5) [below of=4]{$3$}; 
\node[main] (6) [right of=5] {$4$}; 
\node[main] (7) [right of=6]{$5$}; 
\node[main] (8) [below of=6] {$4$};
\node[main] (9) [below of=7]{$5$};  
\node[main] (10) [right of=9]{$6$};  
\node[main] (11) [below of=10]{$6$};  
\node[main] (12) [right of=11]{$7$};  
\draw[->,black] (12) -- (10) node[midway,above right] {1}; 
\draw[->,black] (10) -- (7) node[midway,above right] {1}; 
\draw[->,black] (11) -- (9) node[midway,above right] {2}; 
\draw[->,black] (9) -- (6) node[midway,above right] {2}; 
\draw[->,black] (8) -- (5) node[midway,above right] {3}; 
\draw[->,black] (5) -- (3) node[midway,above right] {3}; 
\draw[->,black] (3) -- (1) node[midway,above right] {3}; 
\end{tikzpicture}
\vspace{6.0cm}
\caption{The M\oe glin-Waldspurger algorithm on a ladder multisegment.}\label{Fig:MWLadder}
\end{figure}
The dual of $\{(1), (2), (3,4,5), (4,5,6), (6,7) \}$ is $\{(1,2,3,4), (4,5,6), (5,6,7) \}$.
\end{exa}

Note that any simple multisegment will also be a ladder multisegment, so we have seen an additional example of a ladder multisegment in \Cref{Exa:Simple}.

\begin{remark}\label{Lem:PrecedesQuantum}
At each iterative step of the M\oe glin-Waldspurger algorithm and when assigning flows in the network description the set of preceding segments which are chosen form an irreducible ladder multisegment. This follows from \Cref{Prop:IncreasingLength} and the properties discussed throughout this subsection must therefore be satisfied for each iteration.
\end{remark}

Following on from {\bfseries Section \ref{Sec:MW-alg}} in which we defined the M\oe glin-Waldspurger algorithm and the segment generated by an iteration of the algorithm by $\Delta'$, then let us now specify that $\Delta'(\alpha)$ will be the segment generated by the first iteration of the M\oe glin-Waldspurger algorithm. Also let us denote $\alpha - \Delta'(\alpha)$ to be the multisegment produced by the algorithm following the removal of the elements chosen for $\Delta'(\alpha)$, this will also be the multisegment in which the next iteration of the algorithm is carried out on. Thus the dual multisegment will be recursively generated by $$ \tilde{\alpha} = \{ \Delta'(\alpha) , \widetilde{\alpha - \Delta'(\alpha)} \}.$$

\begin{lemma}[\cite{riddlesden2022combinatorial}, Lemma  4.2.14.]\label{Lem:AB}
If $n_{\tilde{\alpha}} + n_{\alpha} = S_{\alpha} + C_{\alpha}$, $\alpha \leq \beta$ and $\tilde{\alpha} \leq \tilde{\beta}$ then $n_{\tilde{\alpha}} = n_{\tilde{\beta}}$, $n_{\alpha} = n_{\beta}$ and $n_{\tilde{\beta}} + n_{\beta} = S_{\beta} + C_{\beta}$.
\end{lemma}

\begin{lemma}[\cite{riddlesden2022combinatorial}, Lemma  4.2.15.] \label{Lem:QuantEquiv}
If $\alpha$ and $\beta$ are ladder multisegments, $\alpha \leq \beta$ and $n_{\alpha} = n_{\beta}$ then $\alpha = \beta$.
\end{lemma}

\begin{lemma}[\cite{riddlesden2022combinatorial}, Lemma  4.2.19.] \label{Lem:IrredQFM}
If $\alpha$ is an irreducible ladder multisegment then $$n_{\tilde{\alpha}} + n_{\alpha} = S_{\alpha} + c_{\alpha}.$$
\end{lemma}

\begin{lemma}[\cite{riddlesden2022combinatorial}, Lemma  4.2.21.] \label{Lem:Quant}
If $\alpha$ is any multisegment and $$n_{\tilde{\alpha}} + n_{\alpha} = S_{\alpha} + c_{\alpha},$$ then $\alpha$ is a ladder multisegment.
\end{lemma}

\begin{corollary}[\cite{riddlesden2022combinatorial}, Corollary 4.2.22.] \label{Cor:Quant}
For an arbitrary multisegment $\alpha$ which is not a ladder multisegment, we have $$n_{\tilde{\alpha}} + n_{\alpha} > S_{\alpha} + C_{\alpha} \geq S_{\alpha} + c_{\alpha},$$ and hence a multisegment $\alpha$ is a ladder multisegment if and only if $$n_{\tilde{\alpha}} + n_{\alpha} = S_{\alpha} + C_{\alpha} = S_{\alpha} + c_{\alpha}.$$ 
\end{corollary}

\begin{thm}\label{Thm:Quant}
Let $\alpha$ be a ladder multisegment. If $\beta$ is a multisegment such that $\alpha \leq \beta$ and $\tilde{\alpha} \leq \tilde{\beta}$, then $\alpha = \beta$. 
\end{thm}

\begin{proof}
Since $\alpha$ is a ladder multisegment then $n_{\tilde{\alpha}} + n_{\alpha} = S_{\alpha} + C_{\alpha}$ by \Cref{Cor:Quant}. Using \Cref{Lem:AB}, $n_{\tilde{\beta}} + n_{\beta} = S_{\beta} + C_{\beta}$ and $n_{\alpha} = n_{\beta}$, thus $\beta$ is a ladder multisegment by \Cref{Cor:Quant}. Therefore we have satisfied all hypothesis of \Cref{Lem:QuantEquiv} so $\alpha = \beta$.
\end{proof}

\subsection{Multisegments of Arthur Type}
We will now present the proof that another family of multisegments satisfy the partial ordering relation on multisegments and relate this to a significant conjecture in the local Langlands correspondence that $$ \abovedisplayskip=5pt \text{\emph{ABV-packets for orbits of Arthur type in $GL_n$ are singletons}}. \belowdisplayskip=5pt $$ To do this we must first introduce this notion of Arthur type. A \emph{Langlands parameter of Arthur type} is a Langlands parameter $\phi$ such that $\phi = \phi_{\psi}$ as defined in the book \cite[Section 3.6]{Cunningham1}. One property that this enforces is that the corresponding multisegments must have the property of being symmetric along the zero element $i$, where $\lambda_i$ corresponds to the $q^0$-eigenspace. That is, if we relabel each element in the multisegment $\alpha$ to be such that $i \rightarrow 0$, $i-1 \rightarrow -1$, $i+1 \rightarrow 1$, \dots; then the segment $\Delta = [b,e] \in \alpha$ if and only if the segment $- \Delta = [-e, -b] \in \alpha$.  A key consequence of this restriction is that we are now only considering symmetric irreducible multisegments, so when $S_{\alpha}$ is odd then the description is trivial since $0$ will be a central value. Alternatively when $S_{\alpha}$ is even, we have to slightly modify the description to be such that the two labellings each side of the symmetry will be $-\frac{1}{2}$ and $\frac{1}{2}$, and any subsequent values will then differ by $1$ as they get further away from the centre. Note this description preserves the structure of their being $1$ between each of the labellings of the eigenvectors, and hence preserves the previously discussed properties. Further, we defined the maximum value of our multisegment to be $e_{\alpha}$ so it will always be true that $S_{\alpha} = 2 e_{\alpha} +1.$

The restriction to only studying those Langlands parameters of Arthur type imposes a further condition on the multisegment $\alpha$ that $\alpha$ must be formed from the union of simple symmetric multisegments as discussed in  \cite[Remark 1.1]{Cunningham2}.

\begin{lemma} \label{Lem:A1}
Let $\alpha$ be an arbitrary multisegment containing a sub-multisegment $\alpha_1$ of the form $$ \abovedisplayskip=5pt \belowdisplayskip=5pt \alpha_1 = \left\{ [-e,b], [-e+1, b+1], \dots, [-b-1, e-1], [-b, e] \right\}.$$ If $\alpha_1$ contains both of the shortest segments containing the minimum and maximum values, $-e$ and $e$, of the multisegment $\alpha$ then it will not be possible to generate $[b,e]$ or a segment containing $b, \dots, e$ from any sub-multisegment other than $\alpha_1$.
\end{lemma}

\begin{proof} 
Firstly, by \Cref{Lem:PrecedesQuantum} when generating a segment for the dual we generate a sub-multisegment $\gamma$ which forms an irreducible ladder multisegment. Thus $\gamma$ must be such that it satisfies $$\abovedisplayskip=5pt \belowdisplayskip=5pt n_{\gamma} + n_{\tilde{\gamma}} - c_{\gamma} = S_{\gamma}, $$ by \Cref{Lem:IrredQFM}. Now $S_{\gamma} \leq S_{\alpha} = 2e+1$, since $\gamma$ is a sub-multisegment of $\alpha$. Also $\gamma$ is irreducible so $c_{\gamma}=1$, and the segment created contains $b, \dots, e$ so $n_{\gamma} \geq e - b +1$. By assumption, $e$ is the maximum value and the shortest segment ending in $e$ has length $e-(-b)+1$ since $\alpha$ is both simple and symmetric, so $L_{\gamma} \geq e+b+1$ by \Cref{Prop:IncreasingLength} which implies $n_{\tilde{\gamma}} \geq L_{\gamma} \geq e+b+1$ from \Cref{Basic}. Using these values then we find that $$ S_{\gamma} = n_{\gamma} +n_{\tilde{\gamma}} - c_{\gamma} \geq (e-b+1) + (e+b+1) -1 = 2e+1. $$ Therefore $S_{\gamma} = 2e+1$, and $n_{\gamma}, n_{\tilde{\gamma}}, L_{\gamma}$ must all be minimal. Therefore by \Cref{Prop:IncreasingLength}, every segment in $\gamma$ must be of the minimal length $e+b+1$, and there must exist $e - b +1$ segments covering the values from $[-e,e]$, so the only possible formation for $\gamma$ is the sub-multisegment $\alpha_1$.
\end{proof}

\begin{lemma}\label{Lem:A1b}
Let $\alpha$ be an arbitrary multisegment containing a sub-multisegment $\alpha_1$ of the form $$ \alpha_1 = \left\{ [-e,b], [-e+1, b+1], \dots, [-b-1, e-1], [-b, e] \right\} = \left\{ \Delta^b, \Delta^{b+1}, \dots, \Delta^{e-1}, \Delta^e \right\}.$$ If $\alpha_1$ contains both of the shortest segments containing the minimum and maximum values, $-e$ and $e$, of the multisegment $\alpha$ then removing copies of $\alpha_1$ will induce an endoscopic decomposition, that is, $$ \alpha = \alpha_1 \sqcup (\alpha - \alpha_1) \,\,\,\, \text{and} \,\,\,\,  \tilde{\alpha} = \widetilde{\alpha_1} \sqcup \widetilde{(\alpha - \alpha_1)}.$$
\end{lemma}
\begin{proof}
In \Cref{Sub:Network}, we present a network implementation for finding the dual using maximum flows. If we now fix a rank $\widetilde{r_{i,j}}$ for which we want to find. Then taking the union $\widetilde{\alpha_1} \sqcup \widetilde{(\alpha - \alpha_1)}$ corresponds to finding the maximum flow through $\alpha_1$ and $(\alpha - \alpha_1)$ individually then summing them. The networks associated to $\alpha_1$ and $(\alpha - \alpha_1)$ will be distinct subnetworks of $\alpha$, therefore the sum of their maximum flows for each $(i,j)$ will hence be less than or equal to that in the network $\alpha$. Another thing to note is that the maximum flows for both $\alpha_1$ and $(\alpha - \alpha_1)$ will give a feasible flow through $\alpha$, thus we can apply the principle of the Ford-Fulkerson algorithm to study whether in-fact the flow through $\alpha$ will be greater than the sum of the flows. To do this we must look for possible augmenting flows. For the case when $i=j$ then equality will be true as the multiplicity of the integers must remain the same. Let us instead consider the case $i<j$, then one immediate consequence is that any augmenting flow must use nodes from both subnetworks, since the maximum flow through the subnetworks has already been pushed. An additional constraint on any augmenting flow is that it must include a forward edge from $\alpha_1$ to $(\alpha - \alpha_1)$ or $(\alpha - \alpha_1)$ to $\alpha_1$. This results in a flow through an integer $m$ contained in the simple multisegment $\alpha_1$ which previously did not have flow passing through it. \\


\begin{figure}[H]
 \hspace*{4cm}
\begin{tikzpicture}[node distance={10mm}, thick, main/.style = {draw=none}, transform canvas={scale=1}] 
\node[] (1) {$\mathbf{\Delta_e:}$}; 
\node[main] (2) [below of=1] {\textbf{\vdots}}; 
\node[main] (3) [below of=2]{$\mathbf{\Delta_n:}$}; 
\node[main] (4) [below of=3] {\textbf{\vdots}}; 
\node[main] (5) [below of=4]{$\mathbf{\Delta_b:}$}; 
\node[main] (6) [right of=5] {$-e$}; 
\node[main] (7) [right of=6]{$\dots$}; 
\node[main] (8) [right of=7] {$q$};
\node[main] (9) [right of=8]{$\dots$};  
\node[main] (10) [right of=9]{$b$};  
\node[main] (11) [above of=7]{$\reflectbox{$\ddots$}$};  
\node[main] (12) [right of=11]{\phantom{$7$}};  
\node[main] (13) [right of=12]{$\reflectbox{$\ddots$}$};  
\node[main] (14) [right of=13]{\phantom{$7$}};  
\node[main] (15) [right of=14]{$\reflectbox{$\ddots$}$};  

\node[main] (16) [above of=12] {$k$}; 
\node[main] (17) [right of=16]{$\dots$}; 
\node[main] (18) [right of=17] {$m$};
\node[main] (19) [right of=18]{$\dots$};  
\node[main] (20) [right of=19]{$n$};  
\node[main] (21) [above of=17]{$\reflectbox{$\ddots$}$};  
\node[main] (22) [right of=21]{\phantom{$7$}};  
\node[main] (23) [right of=22]{$\reflectbox{$\ddots$}$};  
\node[main] (24) [right of=23]{\phantom{$7$}};  
\node[main] (25) [right of=24]{$\reflectbox{$\ddots$}$};  
\node[main] (26) [above of=22] {$-b$}; 
\node[main] (27) [right of=26]{$\dots$}; 
\node[main] (28) [right of=27] {$p$};
\node[main] (29) [right of=28]{$\dots$};  
\node[main] (30) [right of=29]{$e$};  
\draw[black,thick,dashed] ($(26.north west)+(-0.15,0.2)$)  rectangle ($(30.south east)+(0.15,-0.2)$);
\draw[black,thick,dashed] ($(16.north west)+(-0.15,0.2)$)  rectangle ($(20.south east)+(0.15,-0.2)$);
\draw[black,thick,dashed] ($(6.north west)+(-0.15,0.2)$)  rectangle ($(10.south east)+(0.15,-0.2)$);
\end{tikzpicture}
\vspace{4.5cm}
\caption{The simple multisegment, $\alpha_1$.}\label{Fig:SimpleA1}
\end{figure}
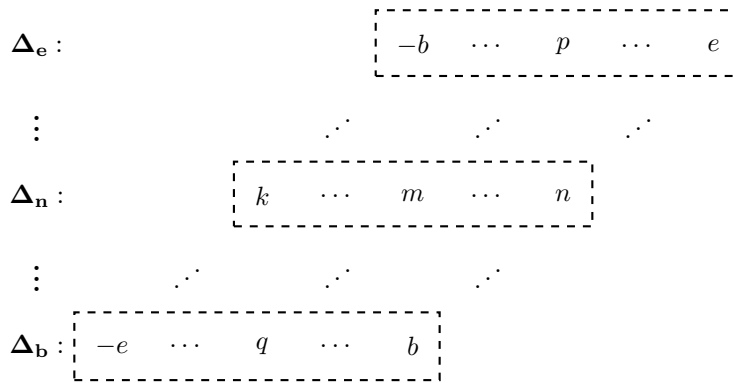

Let us assume that an augmenting flow exists. If we push the augmenting flow containing $m$ associated to $\Delta_n$ then there will exist a flow for $(i,j)$ for which $m$ in $\Delta_n$ is chosen. This flow will be characterised by the irreducible ladder multisegment $\gamma =\{\delta_i, \dots, \delta_m, \dots, \delta_j \}$, where $\delta_a$ represents the segment from which the integer $a$ is chosen from by the flow. By construction at least one of $j>p$ or $i <q$ is true, so we can split it into two individual cases:  

\begin{enumerate}
\item[$\mathbf{j>p:}$] Let us consider a subset of the irreducible ladder multisegment $$ \abovedisplayskip=5pt \belowdisplayskip=5pt \sigma = \{ \delta_m, \dots, \delta_j \},$$ then $\sigma$ will also form an irreducible ladder multisegment. We can then use \Cref{Lem:IrredQFM} to study $S_{\sigma}$, since we know that $$n_{\tilde{\sigma}} + n_{\sigma} = S_{\sigma} + c_{\sigma}.$$ We immediately know a couple of the numerical invariants for $\sigma$, that is, $c_{\sigma} = 1$ and $n_{\sigma}= j-m+1$. Now we also know that the length of the segment $\delta_m$ is $e+b+1$ and using \Cref{Basic}, we find $$ n_{\tilde{\sigma}} \geq L_{\sigma} \geq e+b+1.$$ Therefore, $$ S_{\sigma} = n_{\tilde{\sigma}} + n_{\sigma} - c_{\sigma} \geq (e+b+1) +(j-m+1) - 1 = e+b+j-m+1.$$ 
Further we have the condition $j>p$
$$ S_{\sigma} > e+(b+p)-m+1 = e +(m-k)-m+1 = e-k+1.$$ However given the lowest value of $\sigma$ is $k$ and the maximum value of the original multisegment $\alpha$ for which $\sigma$ is a subset of is $e$ then $S_{\sigma} \leq e-k+1$, hence we have found a contradiction.
\item[$\mathbf{i<q:}$] Let us consider a subset of the irreducible ladder multisegment $$ \abovedisplayskip=5pt \belowdisplayskip=5pt \sigma = \{ \delta_i, \dots, \delta_m \},$$ then $\sigma$ will also form an irreducible ladder multisegment. We can then use \Cref{Lem:IrredQFM} to study $S_{\sigma}$, since we know that $$ n_{\tilde{\sigma}} + n_{\sigma} = S_{\sigma} + c_{\sigma}.$$ We immediately know a couple of the numerical invariants for $\sigma$, that is, $c_{\sigma} = 1$ and $n_{\sigma}= m-i+1$. Now we also know that the length of the segment $\delta_m$ is $e+b+1$ and using \Cref{Basic}, we find $$  n_{\tilde{\sigma}} \geq L_{\sigma} \geq e+b+1.$$ Therefore, $$  S_{\sigma} = n_{\tilde{\sigma}} + n_{\sigma} - c_{\sigma} \geq (e+b+1) +(m-i+1) - 1 = e+b+m-i+1.$$ 
Further we have the condition $i<q$
$$  S_{\sigma} > e+(b-q)+m+1 = e +(n-m)-m+1 = e+n+1. $$ However given the highest value of $\sigma$ is $n$ and the minimum value of the original multisegment $\alpha$ for which $\sigma$ is a subset of is $-e$ then $S_{\sigma} \leq n-(-e)+1$, hence we have found a contradiction. \qedhere
\end{enumerate}
\end{proof}

\vspace{5pt}
\begin{thm} \label{Thm:ManySimpleA=B}
Let $\alpha$ be a multisegment formed by the taking the union of $m$ simple symmetric multisegments. If $\beta$ is a multisegment such that $\alpha \leq \beta$ and $\tilde{\alpha} \leq \tilde{\beta}$, then $\alpha = \beta$. 
\end{thm}

\begin{proof}
Firstly, let $\alpha$ be a multisegment formed by the taking the union of $m$ simple symmetric multisegments, and let us assume that there exists a multisegment $\beta$ such that $\alpha \leq \beta$ and $\tilde{\alpha} \leq \tilde{\beta}$. Then by assumption $\tilde{\alpha} \leq \tilde{\beta}$ so $r_{\tilde{\alpha},i,j} \leq r_{\tilde{\beta},i,j}$ for all $i, j$, where the rank $r_{\tilde{\alpha},i,j}$ denotes the number of appearances of the sequence $i, \dots, j$ in the segments of $\tilde{\alpha}$. In other words, $r_{\tilde{\alpha},i,j}$ is the number of segments $[k,l]$ contained in $\tilde{\alpha}$ such that $k \leq i$ and $j \leq l$ as 
previously discussed. There will exist a maximum value of the multisegment denoted by $e$ then $-e$ will be the minimum value. So when the M\oe glin-Waldspurger algorithm is taken on $\alpha$ then it will choose the segment containing $e$ which is shortest and denoted $\Delta_e$. The segment $\Delta_e$ will be part of a simple symmetric multisegment $\alpha_1$ which forms $\alpha$, and the algorithm will hence generate a segment $[b, e]$ from this simple symmetric multisegment $$\abovedisplayskip=0pt \belowdisplayskip=2pt \alpha_1= \{[-e,b], \dots, [-b,e]\}.$$ 
The segment $\Delta_e=[-b,e]$ will be the shortest segment containing $e$ in $\alpha$ and by \Cref{Lem:A1}, the formation of the multisegment $\alpha_1$ results in it being the only possible contributing factor to $r_{\tilde{\alpha},b,e}$, hence $r_{\tilde{\alpha},b,e}$ simply denotes the number of copies of $\alpha_1$ in $\alpha$.

If we study $r_{\tilde{\alpha},b,e}$ and $r_{\tilde{\beta},b,e}$, then we know that $r_{\tilde{\beta},b,e}$ must be at least $r_{\tilde{\alpha},b,e}$. In order to have $r_{\tilde{\alpha},b,e}<r_{\tilde{\beta},b,e}$, then \Cref{Lem:A1} also implies that this would require us to create shorter segments containing $e$. However to do this in the formation of $\beta$, we would be required to use either union intersection or conjunction. We can immediately rule out the use of conjunction, since this only creates a longer segment. If we now look at union intersection, then the shorter segment which is created will be formed by those values which are repeated by the two segments that the action is taken on. So $e$ must appear in both in order to be in the shorter segment, however if $e$ appears in both then the union intersection will be equal to the shorter segment. Consequently, it is not possible to generate a shorter segment containing $e$ in $\alpha$.

Therefore $r_{\tilde{\alpha},b,e} = r_{\tilde{\beta},b,e}$ and as demonstrated in \Cref{Lem:A1} $\alpha_1$ is the only possible sub-multisegment which can generate $[b,e]$. Additionally, it will not be possible to perform any actions on any of the other segments in $\alpha_1$, because any operation on the segments in $\alpha_1$ would change them, and could no longer be used to form $[b,e]$. Thus each copy of $\alpha_1$ ($\alpha$ could include multiple copies) will also be sub-multisegments used to form $\beta$ since it is the only possible sub-multisegment which contributes to $r_{\tilde{\beta},b,e}$. We can now use \Cref{Lem:A1b} to find $$ \alpha = \alpha_1 \sqcup (\alpha - \alpha_1) \text{\,\,\, and \,\,\,} \beta = \alpha_1 \sqcup (\beta - \alpha_1)$$ will form endoscopic decompositions. The multisegment that remains $(\alpha - \alpha_1)$ following the removal will also be a union of simple symmetric multisegments by construction, thus we can use a recursive argument on the new multisegment $(\alpha - \alpha_1)$, the maximum value $e$, and shortest segment containing it $\Delta_e$ until we reach the case in which the multisegment is formed by a single symmetric multisegment or is empty. If we reach the case the multisegment is formed by a single symmetric multisegment, then we can use the fact that the partial ordering relation is satisfied for a single simple multisegment (\Cref{Thm:Simple}) to show that this should also remain fixed. Therefore, since all $m$ simple symmetric multisegments in $\alpha$ will be used as sub-multisegments in the formation of $\beta$ then $\alpha=\beta$. 
\end{proof}

Therefore we have proved that the partial ordering relation will be satisfied for ABV-packets for orbits of Arthur type. The following corollary proves the significant conjecture: \emph{ABV-packets for orbits of Arthur type in $GL_n$ are singletons}, which was first proposed by Cunningham et al. \cite{Cunningham2}.

\begin{corollary}\label{Cor:Arthur}
ABV-packets for orbits of Arthur type are singletons and consequently, ABV-packets for orbits of Arthur type are A-packets.
\end{corollary}


\section*{Acknowledgements}
I would like to thank my Master's supervisor Prof. Andrew Fiori for both giving me the problem motivating this paper and for all his help in guiding me through the preparation of this article. I would also like to thank the Voganish project for the useful discussions whilst working on the problem.

\bibliographystyle{siam}
\bibliography{CombinatorialApproach-arXiv.bib}

\end{document}